\documentclass[12pt,a4paper]{article}
\usepackage[dvips]{graphicx}
\usepackage{amssymb}
\usepackage{amsmath}
\usepackage{mathrsfs}
\usepackage{color}
\usepackage{comment}
\usepackage{theorem}
\usepackage{ascmac}
\newcommand{\pref}[1]{(\ref{#1})}
\newcommand{\be}{\begin{equation}}
\newcommand{\ee}{\end{equation}}

\newcommand{\qed}{{\unskip\nobreak\hfil\penalty50\quad\null\nobreak\hfil
	$\square$\parfillskip0pt\finalhyphendemerits0\par\medskip}}

\newcommand{\vecb}{\mbox{\boldmath $ b $}}
\newcommand{\vecc}{\mbox{\boldmath $ c $}}

\newcommand{\vecf}{\mbox{\boldmath $ f $}}

\newcommand{\veco}{\mbox{\boldmath $ o $}}

\newcommand{\vecx}{\mbox{\boldmath $ x $}}

\newcommand{\veckappa}{\mbox{\boldmath $ \kappa $}}
\newcommand{\vecnu}{\mbox{\boldmath $ \nu $}}
\newcommand{\vecrho}{\mbox{\boldmath $ \rho $}}
\newcommand{\vectau}{\mbox{\boldmath $ \tau $}}

\newtheorem{thm}{Theorem}[section]
\newtheorem{prop}{Proposition}[section]
\newtheorem{lem}{Lemma}[section]
\newtheorem{cor}{Corollary}[section]

\theorembodyfont{\rmfamily}
\newtheorem{rem}{Remark}[section]
\newtheorem{proof}{\normalfont\itshape Proof.}
\newtheorem{proofa}{\normalfont\itshape Proof of Corollary \ref{Corollary2.1}.}

\title{Interpolation inequalities between the deviation of curvature and the isoperimetric ratio
with applications to geometric flows
\footnote{2010 Mathematics Subject Classification:
53A04,
53C44,
35B40,
35K55}}
\author{Takeyuki Nagasawa
\\
Department of Mathematics, Saitama University
\\
255 Shimo-Okubo, Sakura-ku, Saitama City, Saitama 338-8570, Japan
\\[0.5cm]
Kohei Nakamura
\\
Graduate School of Science and Engineering, Saitama University
\\
255 Shimo-Okubo, Sakura-ku, Saitama City, Saitama 338-8570, Japan}

\date{\today}
\begin{document}
\maketitle
\begin{abstract}
Several inequalities for the isoperimetric ratio for plane curves are derived.
In particular,
we obtain interpolation inequalities between the deviation of curvature and the isoperimetric ratio.
As applications,
we study the large-time behavior of some geometric flows of closed plane curves without a convexity assumption.
\\[0.3cm]
{\it Keywords}: isoperimetric ratio, curvature, interpolation inequalities, geometric flow
\end{abstract}
\section{Introduction}
\par
Let $ \vecf = ( f_1 , f_2 ) \, : \, \mathbb{R} / L \mathbb{Z} \to \mathbb{R}^2 $ be a function such that $ \mathrm{Im} \vecf $ is a closed plane curve with rotation number $ 1 $ and the variable of $ \vecf $ is the arc-length parameter.
The unit tangent vector is $ \vectau = ( f_1^\prime , f_2^\prime ) $.
Let $ \vecnu = ( - f_2^\prime , f_1^\prime ) $ be the inward unit normal vector,
and let $ \veckappa = \vecf^{ \prime \prime } $ be the curvature vector.
The (signed) area $ A $ is given by
\[
	A = - \frac 12 \int_0^L \vecf \cdot \vecnu \, ds .
\]
The curvature $ \kappa = \veckappa \cdot \vecnu $ is positive when $ \mathrm{Im} \vecf $ is convex.
Since the curve has rotation number $ 1 $, the deviation of curvature is
\[
	\tilde \kappa
	=
	\kappa - \frac 1L \int_0^L \kappa \, ds
	=
	\kappa - \frac { 2 \pi } L .
\]
For a non-negative integer $ \ell $,
we set
\[
	I_\ell = L^{ 2 \ell + 1 } \int_0^L | \tilde \kappa^{( \ell )} |^2 ds ,
\]
which is a scale invariant quantity (cf.\
\cite{DKS}).
For $ \ell_1 \leqq \ell_2 \leqq \ell_3 $,
$ I_{ \ell_2 } $ can be interpolated by $ I_{ \ell_1 } $ and $ I_{ \ell_3 } $ using the Gagliardo-Nirenberg inequality.
In this paper, we show that $ I_\ell $ satisfies interpolation inequalities  by use of the isoperimetric ratio and $ I_m $ with $ \ell \leqq m $.
Hereafter we call $ \displaystyle{ \frac { 4 \pi A } { L^2 } } $ the isoperimetric ratio,
not $ \displaystyle{ \frac { L^2 } { 4 \pi A } } $.
We set
\[
	I_{-1} = 1 - \frac { 4 \pi A } { L^2 }
	,
\]
which is also scale invariant,
and is non-negative by the isoperimetric inequality.  
Of course,
$ \tilde \kappa \equiv 0 $ implies $ \mathrm{Im} \vecf $ is a round circle,
which attains the minimum $ I_{-1} = 0 $.
This suggests that $ I_{-1} $ can be dominated by some quantities of $ \tilde \kappa $.
Indeed,
we have
\begin{align*}
	I_{-1}
	= & \
	\frac { L^2 - 4 \pi A } { L^2 }
	=
	\frac 1 { L^2 } \int_0^L \left( - L \vecf \cdot \veckappa + 2 \pi \vecf \cdot \vecnu \right) ds
	\\
	= & \
	- \frac 1L \int_0^L \tilde \kappa ( \vecf \cdot \vecnu ) ds
	\\
	= & \
	- \frac 1L \int_0^L \tilde \kappa \left( \vecf \cdot \vecnu - \frac 1L \int_0^L \vecf \cdot \vecnu \, ds \right) ds
\end{align*}
and
\[
	\left| \vecf \cdot \vecnu - \frac 1L \int_0^L \vecf \cdot \vecnu \, ds \right|
	\leqq L .
\]
Thus it holds that
\[
	0 \leqq
	I_{-1}
	\leqq
	I_0^{ \frac 12 }
	.
\]
However,
since $ \displaystyle{ \vecf \cdot \vecnu - \frac 1L \int_0^L \vecf \cdot \vecnu \, ds = 0 } $ when $ \tilde \kappa \equiv 0 $,
it seems that the above inequality can be improved.
In Section 2,
we will show an improved version
\[
	0 \leqq
	I_{-1}
	\leqq
	\frac { I_0 } { 8 \pi^2 }
\]
in Theorem \ref{Theorem1}.
\par
The converse inequality seems not to hold;
the reason will be clarified in Section 2.
However,
$ I_0 $ can be estimated by use of $ I_{-1} $ with the help of $ \kappa $ and its derivative
\[
	I_0 \leqq I_{-1}^{ \frac 12 }
	\left[ L^3 \int_0^L
	\left\{ \kappa^3 \tilde \kappa + ( \tilde \kappa^\prime )^2 \right\} ds
	\right]^{ \frac 12 }
\]
(see Theorem \ref{Theorem2}).
Combining this inequality and the Gagliardo-Nirenberg inequality,
in Theorem \ref{Theorem3} we will show interpolation inequalities satisfied by $ I_{-1} $,
$ I_\ell $ and $ I_m $ for $ 0 \leqq \ell \leqq m $:
\[
	I_\ell
	\leqq
	C \left( I_{-1}^{ \frac { m - \ell } 2 } I_m + I_{-1}^{ \frac { m - \ell } { m+1 } } I_m^{ \frac { \ell +1 } { m+1 } } \right).
\]
Here the constant $ C $ depends on $ \ell $ and $ m $ but not on $ \tilde \kappa $ nor $ L $.
\par
In the final section we give applications of our inequalities to the analysis of the large-time behavior of some geometric flows of closed plane curves.
Using our inequalities, we can study the analysis of behavior quite easily.
\section{Preliminaries}
\setcounter{equation}{0}
\par
For the vector-valued function $ \vecf = ( f_1 , f_2 ) :\mathbb{R} / L \mathbb{Z} \to \mathbb{R}^2 $,
we define a complex-valued function by
\[
	f = f_1 + i f_2 .
\]
We expand $ f $ by the Fourier series
\[
	f = \sum_{ k \in \mathbb{Z} } \hat f(k) \varphi_k
	,
\]
where
\[
	\varphi_k(s) = \frac 1 { \sqrt L } \exp \left( \frac { 2 \pi i k s } L \right)
	,
	\quad
	\hat f(k) = \int_0^L f \overline{ \varphi_k } ds .
\]
The series $ \displaystyle{ \sum_{ k \in \mathbb{Z} } k^\ell | \hat f(k) |^2 } $ is related to $ \displaystyle{ \left( \frac L { 2 \pi } \right)^\ell \int_0^L \kappa^\ell ds } $.
To see this we need some expression of $ f^{( \ell - 1 )} \overline{ f^\prime } $ in terms of $ \kappa $ and its derivatives.
Set
\[
	F_\ell = f^{(\ell-1)} \overline{ f^\prime } .
\]
\begin{lem}
It holds that
\be
	F_1 = \vecf \cdot \vectau + i \vecf \cdot \vecnu ,
	\quad
	F_\ell = i \kappa F_{ \ell - 1 } + F_{ \ell - 1 }^\prime  \quad \mbox{for} \quad \ell \geqq 2 .
	\label{recurrence}
\ee
\label{Lemma2.1}
\end{lem}
\begin{proof}
Firstly,
since $ \vectau = ( f_1^\prime , f_2^\prime ) $ and $ \vecnu = ( - f_2^\prime , f_1^\prime ) $,
we have
\[
	F_1 = ( f_1 + i f_2 )( f_1^\prime - i f_2^\prime )
	= \vecf \cdot \vectau + i \vecf \cdot \vecnu .
\]
The recurrence relation is derived from
\[
	F_\ell
	=
	f^{ ( \ell - 1 ) } \overline{ f^\prime }
	=
	\left( f^{ ( \ell - 2 ) } \overline{f^\prime} \right)^\prime
	-
	f^{ (\ell - 2 ) } \overline{ f^{ \prime \prime } }
	=
	F_{ \ell - 1 }^\prime
	-
	f^{ ( \ell - 2 )} \overline{ f^\prime }
	f^\prime \overline{ f^{ \prime \prime } }
	=
	F_{ \ell - 1 }^\prime - F_{ \ell - 1 } \overline{ F_3 }
\]
and
\[
	F_3
	=
	f^{ \prime \prime } \overline{ f^\prime }
	=
	( f_1^{ \prime \prime } + i f_2^{ \prime \prime } ) ( f_1^\prime - i f_2^\prime )
	=
	\frac 12 \left( | f^\prime |^2 \right)^\prime
	+ i ( - f_1^{ \prime \prime } f_2^\prime + f_2^{ \prime \prime } f_1^\prime )
	=
	i \vecf^{ \prime \prime } \cdot \vecnu
	=
	i \kappa .
\]
\qed
\end{proof}
\begin{prop}
For $ \ell \geqq 2 $,
it holds that
\be
	\sum_{ k \in \mathbb{Z} } k^\ell | \hat f(k) |^2
	=
	- i^{ 1 - \ell } \left( \frac L { 2 \pi } \right)^\ell
	\int_0^L \kappa F_{ \ell - 1 } ds
	.
	\label{serell}
\ee
\label{prop2.2}
\end{prop}
\begin{proof}
It follows from the recurrence relation in Lemma\ref{Lemma2.1} that
\[
	\int_0^L F_\ell ds = i \int_0^L \kappa F_{ \ell - 1 } ds .
\]
On the other hand,
\[
	\int_0^L F_\ell ds
	=
	\langle f^{( \ell-1 ) } , f^\prime \rangle_{ L^2 }
	=
	- i^\ell \left( \frac { 2 \pi } L \right)^\ell \sum_{ k \in \mathbb{Z} } k^\ell | \hat f(k) |^2
	.
\]
\qed
\end{proof}
\begin{cor}
We have
\begin{align}
	\sum_{ k \in \mathbb{Z} } k | \hat f(k) |^2
	= & \
	\frac { LA } \pi
	\label{ser1}
	,
	\\
	\sum_{ k \in \mathbb{Z} } k^2 | \hat f(k) |^2
	= & \
	\left( \frac L { 2 \pi } \right)^2
	\int_0^L \kappa^0 ds
	=
	\frac { L^3 } { 4 \pi^2 } ,
	\label{ser2}
	\\
	\sum_{ k \in \mathbb{Z} } k^3 | \hat f(k) |^2
	= & \
	\left( \frac L { 2 \pi } \right)^3
	\int_0^L \kappa \, ds
	=
	\frac { L^3 } { 4 \pi^2 } ,
	\label{ser3}
	\\
	\sum_{ k \in \mathbb{Z} } k^4 | \hat f(k) |^2
	= & \
	\left( \frac L { 2 \pi } \right)^4
	\int_0^L \kappa^2 ds
	,
	\label{ser4}
	\\
	\sum_{ k \in \mathbb{Z} } k^5 | \hat f(k) |^2
	= & \
	\left( \frac L { 2 \pi } \right)^5
	\int_0^L \kappa^3 ds
	,
	\label{ser5}
	\\
	\sum_{ k \in \mathbb{Z} } k^6 | \hat f(k) |^2
	= & \
	\left( \frac L { 2 \pi } \right)^6
	\int_0^L \left\{ \kappa^4 + ( \kappa^\prime )^2 \right\} ds.
	\label{ser6}
\end{align}
\label{Corollary2.1}
\end{cor}
\begin{rem}
If $ \kappa > 0 $ everywhere,
it holds that
\[
	\sum_{ k \in \mathbb{Z} } k | \hat f(k) |^2
	=
	\frac L { 2 \pi } \int_0^L
	\frac 1 \kappa \left\{ 1 - \left( f^\prime \bar f \right)^\prime \right\} ds .
\]
In particular if $ \kappa $ is a constant,
then
\[
	\sum_{ k \in \mathbb{Z} } k | \hat f(k) |^2
	=
	\frac L { 2 \pi } \int_0^L \frac { ds } \kappa
	.
\]
\end{rem}  
\begin{proofa}
Since
\[
	\int_0^L F_1 ds = i \int_0^L \vecf \cdot \vecnu \, ds = - 2 i A ,
\]
we obtain
\[
	\sum_{ k \in \mathbb{Z} } k | \hat f(k) |^2
	=
	\frac L { 2 \pi i } \int_0^L f^\prime \bar f \, ds
	=
	\frac L { 2 \pi i } \int_0^L \overline{ F_1 } ds
	=
	\frac { LA } \pi.
\]
Thus \pref{ser1} follows.
The relations \pref{ser2}--\pref{ser6} are consequence of \pref{serell} and \pref{recurrence}.
Indeed,
\[
	\sum_{ k \in \mathbb{Z} } k^2 | \hat f(k) |^2
	=
	- i^{-1} \left( \frac L { 2 \pi } \right)^2 \int_0^L \kappa F_1 ds
	=
	i \left( \frac L { 2 \pi } \right)^2 \int_0^L \kappa
	\left( \vecf \cdot \vectau + i \vecf \cdot \vecnu \right) ds ,
\]
and
\[
	\int_0^L \kappa \vecf \cdot \vectau \, ds
	=
	- \int_0^L \vecf \cdot \vecnu^\prime ds
	=
	\int_0^L \vectau \cdot \vecnu \, ds = 0 ,
\]
\[
	\int_0^L \kappa \vecf \cdot \vecnu \, ds
	=
	\int_0^L \vecf \cdot \vectau^\prime \, ds
	=
	- \int_0^L \vectau \cdot \vectau \, ds = - L .
\]
Thus \pref{ser2} holds.
Since $ \vecf $ is parametrized by the arc-length,
we have $ F_2 = | f^\prime |^2 = \| \vecf^\prime \|^2 = 1 $.
It follows from \pref{recurrence} that
\[
	F_3 = i \kappa ,
	\quad
	F_4 = - \kappa^2 + i \kappa^\prime ,
	\quad
	F_5 = - 3 \kappa \kappa^\prime + i \left( - \kappa^3 + \kappa^{ \prime \prime } \right).
\]
Hence we obtain
\begin{gather*}
	\int_0^L \kappa F_2 ds = \int_0^L \kappa \, ds = 2 \pi ,
	\quad
	\int_0^L \kappa F_3 ds = i \int_0^L \kappa^2 ds ,
	\\
	\int_0^L \kappa F_4 ds = - \int_0^L \kappa^3 ds ,
	\quad
	\int_0^L \kappa F_5 ds = - i \int_0^L \left\{ \kappa^4 + ( \kappa^\prime )^2 \right\} ds .
\end{gather*}
Consequently \pref{ser3}--\pref{ser6} are obtained from \pref{serell}.
\end{proofa}
\begin{cor}
\[
	I_{-1} = \frac { 4 \pi^2 } { L^3 } \sum_{ k \in \mathbb{Z} } k ( k - 1 ) | \hat f(k) |^2 .
\]
\end{cor}
\begin{proof}
We obtain
\[
	I_{-1}
	=
	1 - \frac { 4 \pi A } { L^2 }
	=
	\frac { 4 \pi^2 } { L^3 }
	\left( \frac { L^3 } { 4 \pi^2 } - \frac { LA } \pi \right)
	=
	\frac { 4 \pi^2 } { L^3 }
	\sum_{ k \in \mathbb{Z} } k ( k - 1 ) | \hat f(k) |^2
\]
from \pref{ser2} and \pref{ser1}.
\qed
\end{proof}
\par
Since $ k ( k - 1 ) \geqq 0 $ for $ k \in \mathbb{Z} $,
we obtain the isoperimetric inequality $ I_{-1} \geqq 0 $ from this corollary,
which is essentially the proof by Hurwitz \cite{H}.
\begin{cor}
\[
	I_0 = \frac { 16 \pi^4 } { L^3 } \sum_{ k \in \mathbb{Z} } k^3 ( k - 1 ) | \hat f(k) |^2 .
\]
\end{cor}
\begin{proof}
The assertion is derived as
\begin{align*}
	I_0
	= & \
	L \int_0^L \tilde \kappa^2 ds
	=
	L \int_0^L \kappa \tilde \kappa \, ds
	=
	L \left( \int_0^L \kappa^2 ds - \frac { 2 \pi } L \int_0^L \kappa \, ds \right)
	\\
	= & \
	\frac { 16 \pi^4 } { L^3 }
	\sum_{ k \in \mathbb{Z} } k^3 ( k - 1 ) | \hat f(k) |^2,
\end{align*}
using \pref{ser4} and \pref{ser3}.
\qed
\end{proof}
\par
Since $ k ( k - 1 ) \leqq k^3 ( k - 1 ) $ for $ k \in \mathbb{Z} $,
we have
\be
	I_{-1} \leqq
	\frac { 4 \pi^2 } { L^3 }
	\sum_{ k \in \mathbb{Z} } k^3 ( k - 1 ) | \hat f(k) |^2
	=
	\frac { I_0 } { 4 \pi^2 } .
	\label{notsharp}
\ee
We set
\[
	g = \sum_{ k \in \mathbb{Z} } \sqrt{ k ( k - 1 ) } \hat f(k) \varphi_k ,
\]
and then \pref{notsharp} is
\[
	\| g \|_{ L^2 }^2
	\leqq
	\frac { L^2 } { 4 \pi^2 } \| g^\prime \|_{ L^2 }^2 .
\]
This is Wirtinger's inequality with the best constant.
Therefore it is reasonable to think that \pref{notsharp} cannot be sharpened.
However,
the function $ f $ is not an arbitrary one,
but satisfies $ | f^\prime | \equiv 1 $,
and this suggests that we may be able to improve the constant in \pref{notsharp}.
Indeed,
we can show an improved version by use of \pref{ser2} and \pref{ser3}.
\begin{thm}
We have
\[
	I_{-1} \leqq
	\frac { I_0 } { 8 \pi^2 } .
\]
Equality never holds except the trivial case $ \tilde \kappa \equiv 0 $.
\label{Theorem1}
\end{thm}
\begin{proof}
First observe that \pref{ser2} and \pref{ser3} imply
\[
	\sum_{ k \in \mathbb{Z} } k^2 | \hat f(k) |^2
	=
	\sum_{ k \in \mathbb{Z} } k^3 | \hat f(k) |^2
	.
\]
Hence
\begin{gather*}
	I_{-1}
	=
	\frac { 4 \pi^2 } { L^3 } \sum_{ k \in \mathbb{Z} }
	k ( k^2 - 1 ) | \hat f(k) |^2
	=
	\frac { 4 \pi^2 } { L^3 } \sum_{ k \in \mathbb{Z} }
	k ( k - 1 ) ( k + 1 ) | \hat f(k) |^2
	,
	\\
	I_0
	=
	\frac { 16 \pi^4 } { L^3 } \sum_{ k \in \mathbb{Z} }
	k^2 ( k^2 - 1 ) | \hat f(k) |^2
	=
	\frac { 16 \pi^4 } { L^3 } \sum_{ k \in \mathbb{Z} }
	k^2 ( k - 1 ) ( k+1 )| \hat f(k) |^2 .
\end{gather*}
Consequently we obtain
\[
	\frac { I_0 } { 8 \pi^2 } - I_{-1}
	=
	\frac { 2 \pi^2 } { L^3 }
	\sum_{ k \in \mathbb{Z} }
	k ( k - 2 ) ( k - 1 ) ( k +1 ) | \hat f(k)|^2
	\geqq 0 ,
\]
because $ k ( k - 2 ) ( k - 1 ) ( k + 1 ) \geqq 0 $ for $ k \in \mathbb{Z} $.
\par
For the equality case, assume that $ f $ satisfies $ \displaystyle{ I_{-1} = \frac { I_0 } { 8 \pi^2 } } $.
It follows from the previous paragraph that
\[
	f = \sum_{ k=-1 }^2 \hat f(k) \varphi_k .
\]
Since $ | f^\prime |^2 \equiv 1 $ and since $ \varphi_k \overline{ \varphi_\ell } = L^{ \frac 12 } \varphi_{ k - \ell } $,
we have
\begin{align*}
	1 = & \
	\left| \sum_{ k=-1}^2 \frac { 2 \pi i k } L \hat f(k) \varphi_k \right|^2
	=
	\frac { 4 \pi^2 } { L^2 } \sum_{ k, \ell = - 1 }^2
	k \ell \hat f(k) \overline{ \hat f(\ell) } \varphi_k \overline{ \varphi_\ell }
	\\
	= & \
	\frac { 4 \pi^2 } { L^{ \frac 32 } }
	\sum_{ m = -3 }^3
	\sum_{ \substack{ k - \ell = m \\[1pt] -1 \leqq k \leqq 2 , \, -1 \leqq \ell \leqq 2} }
	k \ell \hat f(k) \overline{ \hat f(\ell) } \varphi_m
	.
\end{align*}
From this we find that,
in particular,
\[
	0
	=
	\sum_{ \substack{ k - \ell = 3 \\[1pt] -1 \leqq k \leqq 2 , \, -1 \leqq \ell \leqq 2} }
	k \ell \hat f(k) \overline{ \hat f(\ell) }
	=
	- 2 \hat f_2 \overline{ \hat f_{-1} } .
\]
On the other hand,
it follows from \pref{ser2} and \pref{ser4} that
\[
	0 = \sum_{ k=-1 }^2 k^2 ( k - 1 ) | \hat f(k) |^2
	=
	- 2 | \hat f_{-1} |^2 + 4 | \hat f_2 |^2 .
\]
Therefore $ \hat f_{-1} = \hat f_2 = 0 $.
Consequently $ f = \hat f_0 \varphi_0 + \hat f_1 \varphi_1 $,
which implies $ \mathrm{Im} \vecf $ is a round circle.
Hence $ \tilde \kappa \equiv 0 $.
\qed
\end{proof}
\par
We remark that it is impossible to show
\[
	k^3 ( k - 1 ) \leqq C k ( k - 1 ) \quad \mbox{for} \quad k \in \mathbb{Z} ,
\]
and this implies that there is no hope to see $ I_0 \leqq C I_{-1} $.
Thereupon we give an estimate of $ I_0 $ in terms of $ I_{-1} $ with the help of $ \kappa $ and its derivative.
\begin{thm}
The integral $ \displaystyle{ \int_0^L \left\{ \kappa^3 \tilde \kappa + ( \tilde \kappa^\prime )^2 \right\} ds } $ is non-negative,
and it holds that
\[
	I_0 \leqq I_{-1}^{ \frac 12 }
	\left[ L^3 \int_0^L
	\left\{ \kappa^3 \tilde \kappa + ( \tilde \kappa^\prime )^2 \right\} ds
	\right]^{ \frac 12 }
	.
\]
Equality never holds except the trivial case $ \tilde \kappa \equiv 0 $.
\label{Theorem2}
\end{thm}
\begin{proof}
By Cauchy's inequality we have
\begin{align*}
	I_0
	\leqq & \
	\frac { 16 \pi^4 } { L^3 }
	\left\{
	\sum_{ k \in \mathbb{Z} } k ( k - 1 ) | \hat f(k) |^2 \right\}^{ \frac 12 }
	\left\{
	\sum_{ k \in \mathbb{Z} } k^5 ( k - 1 ) | \hat f(k) |^2 \right\}^{ \frac 12 }
	\\
	= & \
	\frac { 8 \pi^3 } { L^{ \frac 32 } } I_{-1}^{ \frac 12 }
	\left\{
	\sum_{ k \in \mathbb{Z} } k^5 ( k - 1 ) | \hat f(k) |^2 \right\}^{ \frac 12 }
	,
\end{align*}
and \pref{ser5} and \pref{ser6} show that
\begin{align*}
	\sum_{ k \in \mathbb{Z} } k^5 ( k - 1 ) | \hat f(k) |^2
	= & \
	\left( \frac L { 2 \pi } \right)^6
	\int_0^L
	\left\{ \kappa^4 + ( \kappa^\prime )^2 - \frac { 2 \pi } L \kappa^3 \right\}
	ds
	\\
	= & \
	\left( \frac L { 2 \pi } \right)^6
	\int_0^L
	\left\{ \kappa^3 \tilde \kappa + ( \tilde \kappa^\prime )^2 \right\}
	ds
	.
\end{align*}
\par
Assume that $ f $ satisfies the equality case in the assertion.
It follows from the equality condition of Cauchy's inequality that $ \hat f_k = 0 $ except $ k = 0 $ and $ 1 $.
Consequently $ \mathrm{Im} \vecf $ is a round circle,
and $ \tilde \kappa \equiv 0 $.
\qed
\end{proof}
\section{Interpolation inequalities}
\setcounter{equation}{0}
\par
In this section we derive several interpolation inequalities from Theorem \ref{Theorem2}.
\begin{thm}
Let $ 0 \leqq \ell \leqq m $.
There exists a positive constant $ C = C ( \ell,m ) $ independent of $ L $ such that
\[
	I_\ell
	\leqq
	C \left( I_{-1}^{ \frac { m - \ell } 2 } I_m + I_{-1}^{ \frac { m - \ell } { m+1 } } I_m^{ \frac { \ell +1 } { m+1 } } \right)
\]
holds.
\label{Theorem3}
\end{thm}
\begin{proof}
When $ \ell = m $,
the assertion is clear.
\par
Let $ \ell < m $.
Then $ m \geqq 1 $.
The Gagliardo-Nirenberg inequality shows
\be
	\left( L^{ ( j + 1 ) p - 1 } \int_0^L | \tilde \kappa^{(j)} |^p ds \right)^{ \frac 1p }
	\leqq
	C(j,m,p) I_m^{ \frac 1 { 2m } \left( j - \frac 1p + \frac 12 \right) }
	I_0^{ \frac 12 \left\{ 1 - \frac 1m \left( j - \frac 1p + \frac 12 \right) \right\} }
	\label{GN}
\ee
for $ p \geqq 2 $ and $ j \leqq m $.
Here $ C ( j,m,p ) $ is independent of $ L $.
Combining this and Wirtinger's inequality $ \displaystyle{ I_0 \leqq \frac { I_1 } { 4 \pi^2 } } $,
we have
\[
	L^3 \int_0^L \tilde \kappa^4 ds
	\leqq
	C I_1^{ \frac 12 } I_0^{ \frac 32 }
	\leqq
	C I_1^2
	,
	\quad
	L_2 \int_0^L | \tilde \kappa |^3 ds
	\leqq
	C I_1^{ \frac 14 } I_0^{ \frac 54 }
	\leqq
	C I_1^{ \frac 32 }
	.
\]
Therefore
\begin{align*}
	L^3 \int_0^L \kappa^3 \tilde \kappa \, ds
	= & \
	L^3 \int_0^L \left( \tilde \kappa + \frac { 2 \pi } L \right)^3 \tilde \kappa \, ds
	\\
	\leqq & \
	C \left(
	L^3 \int_0^L \tilde \kappa^4 ds
	+
	L^2 \int_0^L | \tilde \kappa |^3 ds
	+
	L \int_0^L \tilde \kappa^2 ds
	\right)
	\\
	\leqq & \
	C \left( I_1^2 + I_1^{ \frac 32 } + I_1 \right)
	.
\end{align*}
Consequently the inequality in Theorem \ref{Theorem2} implies
\[
	I_0
	\leqq
	C I_{-1}^{ \frac 12 }
	\left( I_1 + I_1^{ \frac 12 } \right) ,
\]
which is the assertion with $ \ell = 0 $,
$ m = 1 $.
\par
Putting $ p = 2 $ in \pref{GN},
we have
\be
	I_j \leqq C( j,m ) I_m^{ \frac jm } I_0^{ 1 - \frac jm } .
	\label{GN2}
\ee
Combining these with $ j = 1 $ and Young's inequality,
we have
\[
	I_0
	\leqq
	C I_{-1}^{ \frac 12 }
	\left( I_m^{ \frac 1m } I_0^{ 1 - \frac 1m } + I_m^{ \frac 1 { 2m } } I_0^{ \frac 12 \left( 1 - \frac 1m \right) } \right)
	\leqq
	\epsilon I_0
	+
	C_\epsilon \left( I_{-1}^{ \frac m2 } I_m + I_{-1}^{ \frac m { m+1 } } I_m^{ \frac 1 { m+1 } } \right)
	,
\]
where $ \epsilon $ is an arbitrary positive number.
Consequently we obtain
\[
	I_0
	\leqq
	C \left( I_{-1}^{ \frac m2 } I_m + I_{-1}^{ \frac m { m+1 } } I_m^{ \frac 1 { m+1 } } \right)
	,
\]
which is the assertion with $ \ell = 0 $,
$ m \geqq 2 $.
\par
Let $ \ell \geqq 1 $.
Using the above inequality  and \pref{GN2} with $ j = \ell $,
we obtain
\[
	I_\ell
	\leqq
	C I_m^{ \frac \ell m } 
	\left( I_{-1}^{ \frac m2 } I_m + I_{-1}^{ \frac m { m+1 } } I_m^{ \frac 1 { m+1 } } \right)^{ 1 - \frac \ell m }
	\leqq
	C \left( I_{-1}^{ \frac { m - \ell } 2 } I_m + I_{-1}^{ \frac { m - \ell } { m+1 } } I_m^{ \frac { \ell +1 } { m+1 } } \right)
	.
\]
\qed
\end{proof}
\section{Applications to geometric flows}
\setcounter{equation}{0}
\par
We give applications of our inequalities to the asymptotic analysis of geometric flows of closed plane curves.
One of the flows is a curvature flow with a non-local term firstly studied by Jiang-Pan \cite{JP},
and another is the area-preserving curvature flow considered by Gage \cite{G}.
If the initial curve is convex,
then the flows exist for all time keeping the convexity,
and the curve approaches a round circle,
this was shown in \cite{JP,G}.
The local existence of flows without a convexity assumption was shown by \v{S}ev\v{c}ovi\v{c}-Yazaki \cite{SY}.
However,
the large-time behavior for this case is still open,
and seems to may occur
finite-time blow-up for some non-convex initial curve \cite{M},
and the global existence for another initial non-convex curve \cite{SY}.
Escher-Simonett \cite{ES} showed the global existence and investigated the large-time behavior of the area-preserving curvature flow for initial data close to a circle and without a convexity assumption.
In this section,
we investigate the large-time behavior of the flow without a convexity assumption {\it assuming} the global existence.
As shown below, our method is quite easy thanks to our inequalities.
\subsection{A curvature flow with a non-local term}
Firstly we consider the large-time behavior of geometric flow
\be
	\partial_t \vecf = \veckappa - \frac L { 2A } \vecnu
	\label{JiangPan}
\ee
of closed plane curves.
Assume that $ \displaystyle{ \vecf \, : \, \bigcup_{ t \geqq 0 } \left( \mathbb{R} / L(t) \mathbb{Z} \times \{t \} \right) \to \mathbb{R}^2 } $ is a global solution with initial rotation number 1.
Along the flow,
the (signed) area $ A $ and the isoperimetric ratio $ \displaystyle{ \frac { 4 \pi A } { L^2 } } $ are non-decreasing,
if the initial (signed) area is positive.
Indeed,
\be
	\frac { d A } { dt } =
	- \int_0^L \partial_t \vecf \cdot \vecnu \, ds
	=
	\int_0^L \left( - \kappa + \frac L { 2A } \right) ds
	=
	\frac { - 4 \pi A + L^2 } { 2A } \geqq 0 ,
	\label{dA/dt}
\ee
\begin{align}
	\frac d { dt } \frac { 4 \pi A } { L^2 }
	= & \
	\frac { 4 \pi } { L^2 }\frac { d A } { dt }
	-
	\frac { 8 \pi A } { L^3 } \frac { d L } { dt }
	=
	\frac { 8 \pi A } { L^3 }
	\int_0^L \partial_t \vecf \cdot
	\left( - \frac L { 2A } \vecnu + \veckappa \right)
	ds
        \nonumber
	\\
	= & \
	\frac { 8 \pi A } { L^3 }
	\int_0^L \left( \kappa - \frac L { 2A } \right)^2 ds
	\geqq 0 .
	\label{d{iso.ratio}/dt}
\end{align}
Since
\be
	\frac { d L } { dt }
	=
	- \int_0^L \partial_t \vecf \cdot \veckappa \, ds
	=
	- \int_0^L \kappa^2 ds
	+ \frac { \pi L } A
	,
	\label{dL/dt}
\ee
we find that $ L $ is non-increasing by Gage's inequality if $ \mathrm{Im} \vecf $ is convex.
Here we do not assume convexity.
\begin{thm}
Assume that $ \vecf $ is a global solution of {\rm \pref{JiangPan}} such that the initial rotation number is $ 1 $ and the initial {\rm (}signed{\rm )} area is positive.
Then $ \mathrm{Im} \vecf $ converges to a circle exponentially as $ t \to \infty $ in the following sense.
There exist $ C > 0 $ and $ \lambda > 0 $ independent of $ t $ such that
\[
	0 \leqq
	I_{-1} (t)
	\leqq
	C e^{ - \lambda t }.
\]
Furthermore,
$ A(t) $ and $ L(t) $ converge to positive constants,
say $ A_\infty $ and $ L_\infty $ respectively,
as $ t \to \infty $.
They satisfy $ 4 \pi A_\infty = L_\infty $,
and
\[
	0 \leqq A_\infty - A(t) \leqq C e^{ - \lambda t } ,
	\quad
	| L(t) - L_\infty | \leqq C e^{ - \lambda t } .
\]
\label{Theorem4}
\end{thm}
\begin{proof}
It follows from \pref{dA/dt} that
\[
	\frac d { dt } A^2 = L^2 - 4 \pi A \geqq 0 .
\]
The second derivative is
\begin{align}
	\frac { d^2 } { dt^2 } A^2
	= & \
	\frac d { dt } ( L^2 - 4 \pi  A )
	=
	2L \frac { dL } { dt } - 4 \pi \frac { dA } { dt }
	\nonumber
	\\
	= & \
	\int_0^L \partial_t \vecf \cdot \left( - 2L \veckappa + 4 \pi \vecnu \right) ds
	=
	\int_0^L \left( \kappa - \frac L { 2A } \right) ( - 2L \tilde \kappa ) \, ds
	\nonumber
	\\
	= & \
	- 2L \int_0^L \tilde \kappa^2 ds
	\leqq 0
	.
	\label{d^2 A/dt^2}
\end{align}
Therefore
\[
	0 \leqq \frac d { dt } A^2 \leqq \left. \frac d { dt } A^2 \right|_{ t=0 } .
\]
We put $ \displaystyle{ C_0 = \left. \frac d { dt } A^2 \right|_{ t=0 } } $.
Since the initial (signed) area is positive,
so is $ A $,
and
\[
	A^2	\leqq C_0 t + A (0)^2 .
\]
Hence
\be
	\int_0^t \frac { dt } A \geqq \int_0^t \frac { dt } { \sqrt { C_0 t + A(0)^2 } } \to \infty \quad \mbox{as} \quad  t \to \infty .
	\label{intA}
\ee
It follows from \pref{d{iso.ratio}/dt} that
\[
	\frac d { dt } I_{-1}
	=
	- \frac { 8 \pi A } { L^3 } \int_0^L
	\left( \tilde \kappa + \frac { 2 \pi } L - \frac L { 2A } \right)^2 ds
	=
	- \frac { 2 \pi } A I_{-1}^2
	- \frac { 8 \pi A } { L^3 } \int_0^L \tilde \kappa^2 ds .
\]
Solving the differential inequality $ \displaystyle{ \frac d { dt } I_{-1} \leqq - \frac { 2 \pi } A I_{-1}^2 } $,
we have
\[
	0 \leqq
	I_{-1}
	\leqq
	\frac { I_{-1} (0) }
	{ \displaystyle{ 1 + 2 \pi I_{-1} (0) \int_0^t \frac { dt } A } }
	\to 0
	\quad ( t \to \infty ) .
\]
Also,
\pref{d^2 A/dt^2} and Theorem \ref{Theorem1} give us
\[
	\frac d { dt } ( L^2 - 4 \pi A )
	+
	\frac { 16 \pi^2 } { L^2 }
	( L^2 - 4 \pi A )
	\leqq
	\frac d { dt } ( L^2 - 4 \pi A )
	+
	2 L \int_0^L \tilde \kappa^2 ds
	= 0
	.
\]
Using $ I_{-1} \to 0 $ as $ t \to \infty $ and \pref{intA},
we obtain
\[
	0
	\leqq
	L^2 - 4 \pi A
	\leqq
	( L(0)^2 - 4 \pi A(0) )
	\exp \left( - 16 \pi^2 \int_0^t \frac { dt } { L^2 } \right)
	\to 0
	\quad \mbox{as} \quad t \to \infty .
\]
Dividing both sides by $ L^2 $,
we have
\[
	0 \leqq
	I_{-1}
	\leqq
	C \frac { 16 \pi^2 } { L^2 }
	\exp \left( - 16 \pi^2 \int_0^t \frac { dt } { L^2 } \right)
	=
	- C \frac d { dt }
	\exp \left( - 16 \pi^2 \int_0^t \frac { dt } { L^2 } \right) .
\]
Integrating this with respect to $ t $,
we obtain
\[
	0 \leqq
	\int_0^t I_{-1} dt
	\leqq
	C \left\{ 1 - 
	\exp \left( - 16 \pi^2 \int_0^t \frac { dt } { L^2 } \right)
	\right\}
	\leqq C .
\]
Integrating \pref{dA/dt},
we have
\[
	A = A(0) + \int_0^t \frac { L^2 } { 2A } I_{-1} dt
	\leqq
	C .
\]
Since $ A $ is non-decreasing and bounded from above,
there exists a finite limit $ \displaystyle{ A_\infty = \lim_{ t \to \infty } A \in ( 0 , \infty ) } $.
Hence $ \displaystyle{ \lim_{ t \to \infty } L =  2 \sqrt{ \pi A_\infty }   \in ( 0 , \infty ) } $ exists,
say $ L_\infty $.
Consequently,
we obtain
\begin{gather*}
	0 \leqq
	L^2 - 4 \pi A
	\leqq
	( L(0)^2 - 4 \pi A(0) )
	\exp \left( - 16 \pi^2 \int_0^t \frac { dt } { L^2 } \right)
	\leqq
	C e^{ - \lambda t }
	,
	\\
	0 \leqq I_{-1}
	= \frac { L^2 - 4 \pi A } { L^2 }
	\leqq
	C e^{ - \lambda t }
	,
	\\
	0 \leqq
	A_\infty - A
	\leqq
	\frac 1 { A_\infty + A } \int_t^\infty \frac d { dt } A^2 dt
	=
	\frac 1 { A_\infty + A } \int_t^\infty ( L^2 - 4 \pi A ) \, dt
	\leqq
	C e^{ - \lambda t }
	,
	\\
	| L - L_\infty |
	=
	\frac { | L^2 - 4 \pi A + 4 \pi ( A - A_\infty ) | } { L + L_\infty }
	\leqq
	C e^{ - \lambda t }
\end{gather*}
for some $ C > 0 $ and $ \lambda > 0 $.
\qed
\end{proof}
\par
Observe that the equation which $ \vecf $ satisfies is
\[
	\partial_t \vecf = \partial_s^2 \vecf - \frac L { 2A } R \partial_s \vecf ,
\]
where
\[
	R = \left(
	\begin{array}{rr}
	0 & -1 \\
	1 & 0
	\end{array}
	\right) .
\]
Since this is a parabolic equation with a non-local term,
$ \vecf $ is smooth for $ t > 0 $ as long as the solution exists.
Hence by shifting the initial time,
we may assume that the initial data is smooth.
Next we consider the decay of $ I_\ell $ as $ t \to \infty $ along this flow.
\begin{thm}
Let $ \vecf $ be as in Theorem \ref{Theorem4}.
For each $ \ell \in \mathbb{N} \cup \{ 0 \} $,
there exist $ C_\ell > 0 $ and $ \lambda_\ell > 0 $ such that
\[
	I_\ell (t) \leqq C_\ell e^{ - \lambda_\ell t } .
\]
\label{Theorem5}
\end{thm}
\begin{proof}
We initially consider the behavior of $ I_0 $.
Since
\[
	I_0
	=
	L \int_0^L \left( \kappa - \frac { 2 \pi } L \right)^2 ds
	=
	L \int_0^L \kappa^2 ds - 4 \pi^2
	,
\]
we have
\begin{align*}
	\frac d { dt } I_0
	= & \
	\frac { dL } { dt } \int_0^L \kappa^2 ds
	+
	L \frac d { dt } \int_0^L \kappa^2 ds
	\\
	= & \
	\int_0^L
	\left\{ - \veckappa \int_0^L \kappa^2 ds
	+
	L \left( 2 \nabla_s^2 \veckappa + \| \veckappa \|_{ \mathbb{R}^2 }^2 \veckappa \right)
	\right\}
	\cdot
	\partial_t \vecf
	\, ds
	\\
	= & \
	\int_0^L
	\left\{ - \kappa \int_0^L \kappa^2 ds
	+
	L \left( 2 \partial_s^2 \kappa + \kappa^3 \right)
	\right\}
	\left( \kappa - \frac L { 2A } \right)
	ds
	\\
	= & \
	- \left( \int_0^L \kappa^2 ds \right)^2
	+
	\frac { \pi L } A \int_0^L \kappa^2 ds
	+
	L \int_0^L
	\left\{
	- 2 ( \partial_s \kappa )^2
	+ \kappa^4
	- \frac L { 2A } \kappa^3
	\right\} ds
	\\
	= & \
	- 2 L \int_0^L ( \partial_s \tilde \kappa )^2 ds
	+
	L \left\{ \int_0^L \kappa^4 ds - \frac 1L \left( \int_0^L \kappa^2 ds \right)^2 \right\}
	-
	\frac { L^2 } { 2A } \int_0^L \kappa^2 \tilde \kappa \, ds
	\\
	= & \
	- 2 L \int_0^L ( \partial_s \tilde \kappa )^2 ds
	- \left( \int_0^L \tilde \kappa^2 ds \right)^2
	\\
	& \quad
	+ \,
	\int_0^L
	\left\{ L \tilde \kappa^4
	+ \left( 8 \pi - \frac { L^2 } { 2A } \right) \tilde \kappa^3
	+ \left( \frac { 16 \pi^2 } L - \frac { 2 \pi L } A \right) \tilde \kappa^2
	\right\}
	ds
	.
\end{align*}
By virtue of Gagliardo-Nirenberg's inequality,
we have
\begin{align}
	\frac d { dt } I_0 + \frac 1 { L^2 } I_0^2 + \frac 2 { L^2 } I_1
	\leqq & \
	\frac C { L^2 }
	\int_0^L \left(
	L^3 \tilde \kappa^4
	+
	L^2 | \tilde \kappa |^3
	+
	L \tilde \kappa^2
	\right) ds
	\nonumber
	\\
	\leqq & \
	\frac C { L^2 }
	\left(
	I_1^{ \frac 12 } I_0^{ \frac 32 }
	+
	I_1^{ \frac 14 } I_0^{ \frac 54 }
	+
	I_0
	\right)
	\label{dI_0/dt}
	.
\end{align}
Applying Young's and Wirtinger's inequalities,
Theorem \ref{Theorem1}
and Theorem \ref{Theorem3},
we have
\begin{align*}
	I_1^{ \frac 12 } I_0^{ \frac 32 }
	\leqq & \
	\epsilon I_1 + C_\epsilon I_0^3 ,
	\\
	I_1^{ \frac 14 } I_0^{ \frac 54 }
	\leqq & \
	\epsilon I_1 + C_\epsilon I_0^{ \frac 53 }
	\leqq
	\epsilon ( I_1 + I_0 ) + C_\epsilon I_0^3
	\leqq
	C \epsilon I_1 + C_\epsilon I_0^3
	,
	\\
	I_0
	\leqq & \
	I_{-1}^{ \frac 12 } \left( I_1 + I_1^{ \frac 12 } \right)
	\leqq
	\left( I_{-1}^{ \frac 12 } + \epsilon \right) I_1 + C_\epsilon I_{-1}
	\\
	\leqq & \
	\left( I_{-1}^{ \frac 12 } + \epsilon \right) I_1 + C_\epsilon e^{ - \lambda t }
\end{align*}
for any $ \epsilon > 0 $.
By Theorem \ref{Theorem4},
for each $ \epsilon > 0 $
there exists $ T_0 > 0 $ such that
$ I_{-1} \leqq \epsilon $ for $ t \geqq T_0 $.
Therefore taking $ \epsilon $ sufficiently small,
we find that there exist some constants $ C_1 $,
$ C_2 $,
and $ C_3 $ such that
\be
	\frac d { dt } I_0 + \frac 1 { L^2 } I_0^2 + \frac { C_1 } { L^2 } I_1
	\leqq
	\frac { C_2 } { L^2 } I_0^3 + \frac { C_3 } { L^2 } e^{ - \lambda t }
	\label{dI_0/dt--2}
\ee
for $ t \geqq T_0 $.
It is obvious that there exists $ T_1 \geqq T_0 $ satisfying
\[
	\int_{ T_1 }^\infty \frac { C_3 } { L^2 } e^{ - \lambda t } dt < \frac 1 { 2 C_2 } .
\]
Furthermore, there exists $ T_2 \geqq T_1 $ such that $ \displaystyle{ I_0 ( T_2 ) < \frac 1 { 2 C_2 } } $,
because we have
\[
	\int_0^\infty I_0 dt \leqq C
\]
by integrating \pref{d^2 A/dt^2}.
We would like to show $ \displaystyle{ I_0 ( t ) < \frac 1 { C_2 } } $ for $ t \geqq T_2 $.
To do this,
we argue by contradiction.
Then there exists $ T_3 > T_2 $ such that
\[
	I_0 ( t ) < \frac 1 { C_2 } \mbox{ for $ t \in [ T_2 , T_3 ) $,
	and } I_0 ( T_3 ) = \frac 1 { C_2 } .
\]
It follow from \pref{dI_0/dt--2} that
\[
	\frac d { dt } { I_0 }
	\leqq
	\frac { C_3 } { L^2 } e^{ - \lambda t }
\]
for $ t \in [ T_2 , T_3 ] $.
Hence
\[
	I_0 ( T_3 )
	=
	I_0 ( T_2 )
	+
	\int_{ T_2 }^{ T_3 } \frac d { dt } I_0 dt
	<
	\frac 1 { 2 C_2 }
	+
	\int_{ T_1 }^\infty \frac { C_3 } { L^2 } e^{ - \lambda t } dt
	<
	\frac 1 { C_2 } .
\]
This contradicts $ \displaystyle{ I_0 ( T_3 ) = \frac 1 { C_2 } } $.
Consequently $ I_0 $ is uniformly bounded,
and \pref{dI_0/dt} implies
\[
	\frac d { dt } I_0
	+
	\frac 1 { L^2 } I_0^2 + \frac 2 { L^2 } I_1
	\leqq
	\frac C { L^2 } \left( I_1^{ \frac 12 } I_0^{ \frac 12 } + I_1^{ \frac 14 } I_0^{ \frac 34 } + I_0 \right) .
\]
By Theorem \ref{Theorem3} we have
\[
	I_1^{ \frac 12 } I_0^{ \frac 12 } + I_1^{ \frac 14 } I_0^{ \frac 34 } + I_0
	\leqq
	\epsilon I_1 + C_\epsilon I_0
	\leqq
	\epsilon I_1 + C_\epsilon I_{-1}^{ \frac 12 } \left( I_1 + I_1^{ \frac 12 } \right)
	\leqq
	\left( 2\epsilon + C_\epsilon I_{-1}^{ \frac 12 } \right) I_1
	+
	C_\epsilon I_{-1}
	,
\]
where $ \epsilon $ is an arbitrary positive number.
Taking $ \epsilon $ sufficiently small,
and $ T_0 $ larger if necessary,
we have $ \epsilon + C_\epsilon I_{-1}^{ \frac 12 } < 1 $ for $ t \geqq T_0 $.
With help of Theorem \ref{Theorem1} and Theorem \ref{Theorem4} we find
\[
	\frac d { dt } I_0 + C_4 \frac { I_0 } { L^2 } \leqq \frac { C_5 } { L^2 } I_{-1}
	\leqq
	C_6 e^{ - \lambda t }
\]
for large $ t $.
Thus the assertion for $ \ell = 0 $ with some positive $ \lambda_0 $ has been proved.
\par
Next we show the exponential decay of $ I_\ell $ for $ \ell \in \mathbb{N} $.
Set
\[
	J_{k,p}
	=
	\left\{ L^{ ( 1 + k ) p - 1 } \int_0^L | \partial_s^k \tilde \kappa |^p ds \right\}^{ \frac 1p }
	.
\]
By Gagliardo-Nirenberg's inequality we have
\be
	J_{ k,p } \leqq C J_{m,2}^\theta J_{0,2}^{ 1 - \theta }
	=
	C I_m^{ \frac \theta 2 } I_0^{ \frac { 1 - \theta } 2 }
	\label{J}
\ee
for $ k \in \{ 0 , 1 , \cdots , m \} $,
$ p \geqq 2 $.
Here $ C $ is independent of $ L $,
and $ \theta =  \frac 1m  \left( k - \frac 1p + \frac 12 \right)  \in [ 0,1 ] $.
\par
Now observe that
\begin{align*}
	\frac d { dt } I_\ell
	= & \
	( 2 \ell + 1 ) L^{ 2 \ell } \frac { dL } { dt } \int_0^L ( \partial_s^\ell \tilde \kappa )^2 ds
	+
	L^{ 2 \ell + 1 } \int_0^L ( \partial_s^\ell \tilde \kappa )^2 \partial_t ( ds )
	\\
	& \
	+ \,
	2 L^{ 2 \ell + 1 } \int_0^L ( \partial_s^\ell \tilde \kappa ) ( \partial_t \partial_s^\ell \tilde \kappa ) \, ds
	.
\end{align*}
It follows from \pref{dL/dt} and Theorem \ref{Theorem4} that
\[
	( 2 \ell + 1 ) L^{ 2 \ell } \frac { dL } { dt } \int_0^L ( \partial_s^\ell \tilde \kappa )^2 ds
	=
	( 2 \ell + 1 ) L^{ 2 \ell } \left( - \int_0^L \kappa^2 ds + \frac { \pi L } A \right)
	\int_0^L ( \partial_s^\ell \tilde \kappa )^2 ds
	\leqq
	\frac C { L^2 } I_\ell .
\]
Since $ \partial_t ( ds ) = - \partial_t \vecf \cdot \veckappa \, ds $,
we have
\begin{align}
	L^{ 2 \ell + 1 } \int_0^L ( \partial_s^\ell \tilde \kappa )^2 \partial_t ( ds )
	= & \
	L^{ 2 \ell + 1 } \int_0^L ( \partial_s^\ell \tilde \kappa )^2 \left\{ - \kappa \left( \kappa - \frac L { 2A } \right) \right\} ds
	\nonumber
	\\
	\leqq & \
	\frac { L^{ 2 \ell + 2 } } { 2A } \int_0^L ( \partial_s^\ell \tilde \kappa )^2
	\left( \tilde \kappa + \frac { 2 \pi } L \right) ds
	\nonumber
	\\
	\leqq & \
	\frac { L^{ 2 \ell + 2 } } { 2A } \int_0^L ( \partial_s^\ell \tilde \kappa )^2 \tilde \kappa \, ds
	+
	\frac C { L^2 } I_\ell
	\label{II}
	.
\end{align}
For $ k \in \mathbb{N} \cup \{ 0 \} $ and $ m \in \mathbb{N} $,
let $ P_m^k ( \tilde \kappa ) $ be any linear combination of the type
\[
	P_m^k ( \tilde \kappa )
	=
	\sum_{ i_1 + \cdots + i_m = k } c_{ i_1 , \dots , i_m }
	\partial_s^{ i_1 } \tilde \kappa \cdots \partial_s^{ i_m } \tilde \kappa
\]
with universal,
constant coefficients $ c_{ i_1 , \dots , i_m } $.
Similarly we define $ P_0^k $ as a universal constant.
We can show
\[
	\partial_t \partial_s^k \tilde \kappa
	=
	\partial_s^{ k+2 } \tilde \kappa
	+
	\sum_{ m=0 }^3 L^{ - ( 3 - m ) } P_m^k ( \tilde \kappa )
	+
	\frac LA \sum_{ m=0 }^2 L^{ - ( 2 - m ) } P_m^k ( \tilde \kappa )
\]
by induction on $ k $. 
Therefore we have
\begin{align*}
	&
	2 L^{ 2 \ell + 1 } \int_0^L ( \partial_s^\ell \tilde \kappa ) ( \partial_t \partial_s^\ell \tilde \kappa ) \, ds
	\\
	& \quad
	=
	- \frac 2 { L^2 } I_{ \ell + 1 }
	+
	L^{ 2 \ell + 1 }
	\int_0^L
	( \partial_s^\ell \tilde \kappa )
	\left( \sum_{ m=0 }^3 L^{ - ( 3 - m ) } P_m^\ell ( \tilde \kappa )
	+ \frac LA \sum_{ m=0 }^2 L^{ - ( 2 - m ) } P_m^\ell ( \tilde \kappa )
	\right) ds .
\end{align*}
Since $ P_0^\ell ( \tilde \kappa ) $ is a constant,
\[
	\int_0^L ( \partial_s^\ell \tilde \kappa ) P_0^\ell ( \tilde \kappa ) \, ds = 0 .
\]
Since $ P_1^\ell ( \tilde \kappa ) = c \partial_s^\ell \tilde \kappa $, we have
\begin{align*}
	L^{ 2 \ell + 1 } \left|
	\int_0^L ( \partial_s^\ell \tilde \kappa ) L^{-2} P_1^\ell ( \tilde \kappa ) \, ds \right|
	= & \
	\frac { | c | } { L^2 } I_\ell ,
	\\
	L^{ 2 \ell + 1 } \left|
	\int_0^L ( \partial_s^\ell \tilde \kappa ) \frac LA L^{-1} P_1^\ell ( \tilde \kappa ) \, ds \right|
	\leqq & \
	\frac C { L^2 } I_\ell .
\end{align*}
Hence we obtain
\begin{align}
	&
	\frac d { dt } I_\ell + \frac 2 { L^2 } I_{ \ell + 1 }
	\nonumber
	\\
	& \quad
	\leqq
	\frac C { L^2 } I_\ell
	+
	L^{ 2 \ell + 1 }
	\int_0^L
	( \partial_s^\ell \tilde \kappa )
	\left(
	\sum_{ m=2 }^3 L^{ - ( 3 - m ) } P_m^\ell ( \tilde \kappa )
	+
	\frac LA P_2^\ell ( \tilde \kappa )
	\right)
	ds
	.
	\label{dI_ell/dt}
\end{align}
Here the first term on the last line of \pref{II} is included into $ \displaystyle{ L^{ 2 \ell + 1 }
	\int_0^L
	( \partial_s^\ell \tilde \kappa )
	\frac LA P_2^\ell ( \tilde \kappa ) \,
	ds } $.
\par
Now we estimate each term on the right-hand side of \pref{dI_ell/dt}.
Firstly we have from Theorem \ref{Theorem3}
\[
	\frac C { L^2 } I_\ell
	\leqq
	\frac C { L^2 } \left( I_{-1}^{ \frac 12 } I_{ \ell + 1 } + I_{-1}^{ \frac 1 { \ell + 2 } } I_{ \ell + 1 }^{ \frac {\ell + 1 } { \ell + 2 } }\right)
	\leqq
	\frac C { L^2 }
	\left\{ \left( I_{-1}^{ \frac 12 } + \epsilon \right) I_{ \ell + 1 } + C_\epsilon I_{-1} \right\}
\]
for any $ \epsilon > 0$.
Taking $ \epsilon $ small and $ t $ large,
by Theorem \ref{Theorem4},
the term
$\displaystyle  \frac C { L^2 }
	\left\{ \left( I_{-1}^{ \frac 12 } + \epsilon \right) I_{ \ell + 1 } \right\} $
can be absorbed into the left-hand side of \pref{dI_ell/dt}.
The second term is a function decaying exponentially.
We will estimate the integral of $ L^{ 2 \ell + m - 2 } ( \partial_s^\ell \tilde \kappa ) P_m^\ell ( \tilde \kappa ) $ for $ m = 2 $ and $ 3 $.
\par
We first consider the case $ m = 2 $.
$ P_2^\ell ( \tilde \kappa ) $ is a linear combination of $ ( \partial_s^k \tilde \kappa ) ( \partial_s^{ \ell - k } \tilde \kappa ) $ with $ k = 0 $,
$ \cdots $,
$ \ell $.
By H\"{o}lder's inequality,
we have
\[
	\left| L^{ 2 \ell + 1 } \int_0^L
	( \partial_s^\ell \tilde \kappa ) L^{-1} P_2^\ell ( \tilde \kappa ) \, ds
	\right|
	\leqq
	\sum_{k=0}^\ell \frac C { L^2 } J_{ \ell , 3 } J_{ k,3 } J_{ \ell - k , 3 }
	,
\]
and \pref{J} yields
\[
	J_{ j , 3 } \leqq C I_{ \ell + 1 }^{ \frac { \theta (j,3) } 2 } I_0^{ \frac { 1 - \theta (j,3) } 2 }
	,
	\quad
	\theta ( j,3 ) = \frac { j + \frac 16 } { \ell + 1 } .
\]
Hence applying Young's inequality,
we obtain
\[
	\left| L^{ 2 \ell + 1 } \int_0^L
	( \partial_s^\ell \tilde \kappa ) L^{-1} P_2^\ell ( \tilde \kappa ) \, ds
	\right|
	\leqq
	\frac C { L^2 } I_{ \ell + 1 }^{ \frac { 2 \ell + \frac 12 } { 2 ( \ell + 1 ) } }
	I_0^{ \frac { \ell + \frac 52 } { 2 ( \ell + 1 ) } }
	\leqq
	\frac \epsilon { L^2 } I_{ \ell + 1 }
	+
	\frac { C_\epsilon } { L^2 } I_0^{ \frac { 2 \ell + 5 } 3 }
\]
for any $ \epsilon > 0$.
$ \displaystyle \frac \epsilon { L^2 } I_{ \ell + 1 } $
can be absorbed into the left-hand side of \pref{dI_ell/dt},
and the second one is a function decaying exponentially.
\par
We can estimate for the case $ m = 3 $ similarly.
Indeed,
Since $ P_3^\ell ( \tilde \kappa ) $ is a linear combination of $ \partial_s^j \tilde \kappa $,
$ \partial_s^k \tilde \kappa $,
and $ \partial_s^{ \ell - j - k } \tilde \kappa $,
we have
\begin{align*}
	\left| L^{ 2 \ell + 1 } \int_0^L
	( \partial_s^\ell \tilde \kappa ) P_3^\ell ( \tilde \kappa ) \, ds
	\right|
	&\leqq
	\sum_{m=0}^\ell \sum_{\substack{j+k=m\\ j\geqq 0,k\geqq 0}} \frac C { L^2 } J_{ \ell , 4 } J_{ j,4 } J_{ k,4 } J_{ \ell - j - k , 4 }\\
	&\leqq
	\frac C { L^2 } I_{ \ell + 1 }^{ \frac { 2 \ell + 1 } { 2 ( \ell + 1 ) } } I_0^{ \frac { 2 \ell + 3 } { 2 ( \ell + 1 ) } }
	\leqq
	\frac \epsilon { L^2 } I_{ \ell + 1 }
	+
	\frac { C_\epsilon } { L^2 } I_0^{ 2 \ell + 3 }
\end{align*}
for any $ \epsilon > 0$.
\par
Since $ \displaystyle{ \frac { L^2 } A } $ is uniformly bounded by Theorem \ref{Theorem4},
we have
\[
	\left|
	L^{ 2 \ell + 1 }
	\int_0^L 
	( \partial_s^\ell \tilde \kappa )
	\frac LA
	P_2^\ell ( \tilde \kappa ) \, ds
	\right|
	\leqq C
	\left|
	L^{ 2 \ell + 1 }
	\int_0^L 
	( \partial_s^\ell \tilde \kappa )
	L^{-1}
	P_2^\ell ( \tilde \kappa ) \, ds
	\right|
	\leqq
	\frac \epsilon { L^2 } I_{ \ell + 1 }
	+
	\frac { C_\epsilon } { L^2 } I_0^{ \frac { 2 \ell + 5 } 3 }.
\]
\par
Therefore \pref{dI_ell/dt} and Wirtinger's inequality imply
\[
	\frac d { dt } I_\ell + \frac 1 { C L^2 } I_\ell
	\leqq
	C e^{ - \mu t } ,
\]
which shows the exponential decay of $ I_\ell $.
\qed
\end{proof}
\begin{thm}
Let $ \vecf $ be as in Theorem \ref{Theorem4},
and let $ \displaystyle{ f(s,t) = \sum_{ k \in \mathbb{Z} } \hat f(k) (t) \varphi_k (s) } $ be the Fourier expansion for any fixed $ t > 0 $.
Set
\[
	\vecc (t) = \frac 1 { \sqrt { L(t) } }( \Re \hat f(0) (t) , \Im \hat f(0) (t) ) ,
\]
and define $ r(t) \geqq 0 $ and $ \sigma (t) \in  \mathbb{R} / 2 \pi \mathbb{Z}  $ by
\[	
	\hat f(1) (t) = \sqrt{ L(t) } r(t) \exp \left( i \frac { 2 \pi \sigma (t) } { L(t) } \right) .
\]
Furthermore we set
\[
	\tilde { \vecf } ( \theta , t )
	=
	\vecf ( L(t) \theta - \sigma (t) , t ) ,
	\quad \mbox{for} \quad
	( \theta , t ) \in \mathbb{R} / \mathbb{Z} \times [ 0 , \infty )
	.
\]
Then the following claims hold.
\begin{itemize}
\item[{\rm (1)}]
There exists $ \vecc_\infty \in \mathbb{R}^2 $ such that
\[
	\| \vecc (t) - \vecc_\infty \| \leqq C e^{ - \gamma t } .
\]
\item[{\rm (2)}]
The function $ r(t) $ converges to the constant $ \displaystyle{ \frac { L_\infty } { 2 \pi } } $ exponentially as $ t \to \infty $:
\[
	\left| r(t) - \frac { L_\infty } { 2 \pi } \right|
	\leqq
	C e^{ - \gamma t } .
\]
\item[{\rm (3)}]
There exists $ \sigma_\infty \in \mathbb{R} / 2 \pi \mathbb{Z} $ such that
\[
	| \sigma (t) - \sigma_\infty | \leqq C e^{ - \gamma t } .
\]
\item[{\rm (4)}]
For any $ k \in \{ 0 \} \cup \mathbb{N} $ there exist $ C_k > 0 $ and $ \gamma_k > 0 $ such that 
\[
	\| \tilde { \vecf } ( \cdot , t ) - \tilde { \vecf }_\infty \|_{ C^k ( \mathbb{R} / \mathbb{Z} ) }
	\leqq
	C_k e^{ - \gamma_k t }
	,
\]
where
\[
	\tilde { \vecf }_\infty ( \theta )
	=
	\vecc_\infty + \frac { L_\infty } { 2 \pi } ( \cos 2 \pi \theta , \sin 2 \pi \theta ) .
\]
\item[{\rm (5)}]
For sufficiently large $ t $,
$ \mathrm{Im} \tilde { \vecf }( \cdot , t ) $ is the boundary of a bounded domain $ \Omega (t) $.
Furthermore, there exists $ T_\ast \geqq 0 $ such that $ \Omega (t) $ is strictly convex for $ t \geqq T_\ast $.
\item[{\rm (6)}]
Let $ D_{ r_\infty } ( \vecc_\infty ) $ be the closed disk with center $ \vecc_\infty $ and radius $ r_\infty $.
Then we have 
\[
	d_H ( \overline { \Omega(t) } , D_{ r_\infty } ( \vecc_\infty ) )
	\leqq
	C e^{ - \gamma t } ,
\]
where $ d_H $ is the Hausdorff distance.
\item[{\rm (7)}]
Let $\displaystyle 
	\vecb (t) =
	\frac 1 { A(t) } \iint_{ \Omega (t) } \vecx \, d \vecx$
be the barycenter of $\Omega(t)$. Then we have
\[
\|A(t) ( \vecb (t) - \vecc (t) )\| \leqq C e^{ - \gamma t }.
\]
\end{itemize}
\label{Theorem6}
\end{thm}
\begin{proof}
\begin{itemize}
\item[{\rm (1)}]
First observe that
\[
	\vecc = \frac 1L \int_0^L \vecf \, ds .
\]
Since
\[
	\partial_t \vecf
	=
	\partial_s \left( \partial_s \vecf - \frac L { 2A } R \vecf \right)
	,
\]
we have
\[
	\int_0^L \partial_t \vecf \, ds = \veco .
\]
Therefore the time-derivative of $ \vecc $ is
\begin{align*}
	\frac d { dt } \vecc
	= & \
	\frac 1L \int_0^L \vecf \partial_t ( ds )
	- \frac 1 { L^2 } \frac { d L } { dt } \int_0^L \vecf \, ds
	\\
	= & \
	- \frac 1L \int_0^L \left\{ ( \partial_t \vecf \cdot \veckappa ) - \frac 1L \int_0^L ( \partial_t \vecf \cdot \veckappa ) \, ds \right\}
	\vecf \, ds
	\\
	= & \
	- \frac 1L \int_0^L \left\{ ( \partial_t \vecf \cdot \veckappa ) - \frac 1L \int_0^L ( \partial_t \vecf \cdot \veckappa ) \, ds \right\}
	\left( \vecf - \frac 1L \int_0^L \vecf \, ds \right) ds .
\end{align*}
Since
\[
	\partial_t \vecf \cdot \veckappa =
	\kappa^2 - \frac L { 2A } \kappa
	=
	\tilde \kappa^2 - \frac L { 2A } ( 2 I_{-1} - 1 ) \tilde \kappa - \frac \pi A I_{-1}
\]
decays exponentially as $ t \to \infty $,
and because
\[
	\left\| \vecf - \frac 1L \int_0^L \vecf \, ds \right\|_{ \mathbb{R}^2 }
	\leqq
	L
	\leqq C ,
\]
we find that $ \vecc $ converges exponentially to a vector,
say $ \vecc_\infty $,
as $ t \to \infty $.
Consequently $ \mathrm{Im} \vecf $ converges to a circle with center at $ \vecc_\infty $.
\item[{\rm (2)}]
It follows from Proposition \ref{prop2.2} and Lemma \ref{Lemma2.1} that
\begin{align*}
	\sum_{ k \in \mathbb{Z} } k^\ell ( k - 1 ) | \hat f(k) |^2
	= & \
	- i^{ - \ell } \left( \frac L { 2 \pi } \right)^{ \ell + 1 } \int_0^L \kappa F_\ell ds
	+
	i^{ 1 - \ell } \left( \frac L { 2 \pi } \right)^\ell \int_0^L \kappa F_{ \ell - 1 } ds
	\\
	= & \
	- i^{ - \ell } \left( \frac L { 2 \pi } \right)^{ \ell + 1 } \int_0^L \kappa F_\ell ds
	+
	i^{ - \ell } \left( \frac L { 2 \pi } \right)^\ell \int_0^L F_\ell ds
	\\
	= & \
	- i^{ - \ell } \left( \frac L { 2 \pi } \right)^{ \ell + 1 } \int_0^L \tilde \kappa F_\ell ds
\end{align*}
for $ \ell \geqq 2 $.
Since $ F_2 $ is a constant,
and since $ F_\ell $ with $ \ell \geqq 3 $ is a polynomial function of $ \kappa $ and its derivatives up to the $ (\ell - 3) $rd order,
they are bounded functions of $ ( s , t ) $.
Also $ L $ is bounded and
$ \tilde \kappa $ decays exponentially as $ t \to \infty $.
Therefore when $ \ell $ is odd,
\[
	\left| \sum_{ k \ne 0 , 1 } k^{ \ell + 1 } | \hat f(k) |^2 \right|
	\leqq
	C \left| \sum_{ k \in \mathbb{Z} } k^\ell ( k - 1 ) | \hat f(k) |^2 \right|
	\leqq
	C_\ell e^{ - \gamma_\ell t } .
\]
By the Parseval identity and the Sobolev embedding theorem we have
\[
	\left\| f (\cdot,t) - \hat f(0) (t) \varphi_0 (\cdot) - \hat f(1) (t) \varphi_1 (\cdot) \right\|_{ C^k ( \mathbb{R} / L(t) \mathbb{Z} ) }
	\leqq
	C_k e^{ - \gamma^\prime_k t }
\]
for any $ k $.
Using the expression of $ \mathbb{R}^2 $-valued functions,
we have
\[
	\vecf ( s,t )
	=
	\vecc (t)
	+
	r(t) \left( \cos \frac { 2 \pi ( s + \sigma (t) ) } { L(t) } , \sin \frac { 2 \pi ( s + \sigma (t) ) } { L(t) } \right)
	+
	\vecrho ( s , t )
	,
\]
\[
	\| \vecrho ( \cdot , t ) \|_{ C^k ( \mathbb{R} / L(t) \mathbb{Z} ) }
	\leqq
	C_k e^{ - \gamma^\prime_k t } .
\]
Since
\[
	\partial_s \vecf (s,t)
	=
	\frac { 2 \pi r(t) } { L(t) } \left( - \sin \frac { 2 \pi ( s + \sigma (t) ) } { L(t) } , \cos \frac { 2 \pi ( s + \sigma (t) ) } { L(t) } \right)
	+
	\partial_s \vecrho ( s , t )
	,
\]
we have
\[
	\left| r(t) - \frac { L(t) } { 2 \pi } \right|
	=
	\frac { L(t) } { 2 \pi }
	\left| \frac { 2 \pi r(t) } { L(t) } - 1 \right|
	=
	\frac { L(t) } { 2 \pi }
	\left| \| \partial_s \vecf ( s,t ) - \partial_s \vecrho (s,t) \| - 1 \right|
	\leqq
	C e^{ - \gamma_1 t } .
\]
Therefore $ r(t) $ converges to $ \displaystyle{ r_\infty = \frac { L_\infty } { 2 \pi } } $ exponentially as $ t \to \infty $.
\item[{\rm (3)}]
First we clarify the meaning of $ \partial_t \vecf $,
it is not $ \displaystyle{ \lim_{ h \to 0 } \frac { \vecf ( s , t + h ) - \vecf ( s , t ) } h } $ as one might expect.
The variable $ s $ in $ f( s,t ) $ is an element of $ \mathbb{R} / L( t) \mathbb{Z} $,
on the other hand,
the $ s $ in $ f ( s , t + h ) $ is in $ \mathbb{R} / L(t+h) \mathbb{Z} $.
Hence the above quotient is not well-defined.
To address this, let us introduce a function $ \bar { \vecf } $ on $ \mathbb{R} / 2 \pi \mathbb{Z} \times [ 0 , \infty ) $ given by
\[
	\bar { \vecf } ( u,t ) = \vecf \left( \frac { L(t) u } { 2 \pi } , t \right) .
\]
Then the variable $ u $ is independent of $ t $,
and $ \partial_t \vecf $ is given by
\[
	\partial_t \vecf
	=
	\lim_{ h \to 0 } \frac { \bar { \vecf } ( u, t + h ) - \bar { \vecf } ( u,t ) } h
	.
\]
We define a complex-valued function $ \bar f $ by
\[
	\bar { \vecf } ( u,t ) = (\Re \bar f ( u,t ) , \Im \bar f ( u,t )) ,
\]
and note that the Fourier expansion of $ \bar f $ is
\[
	\bar f ( u,t )
	=
	\sum_{ k \in \mathbb{Z} }
	\hat f (k) (t) \varphi_k \left( \frac { L(t) u } { 2 \pi } \right)
	=
	\sqrt { 2 \pi }
	\sum_{ k \in \mathbb{Z} }
	\frac { \hat f (k) (t) } { \sqrt { L(t) } }
	\phi_k (u)
	,
\]
where
\[
	\phi_k (u) = \frac 1 { \sqrt { 2 \pi } } e^{ iku } .
\]
Therefore we have
\[
	\partial_t \bar f
	=
	\sqrt{ 2 \pi } \sum_{ k \in \mathbb{Z} }
	\frac d { dt } \frac { \hat f (k) (t) } { \sqrt{ L(t) } } \phi_k (u) ,
\]
and
\[
	\int_0^{ 2 \pi }
	| \partial_t \bar f |^2 du
	=
	2 \pi \sum_{ k \in \mathbb{Z} }
	\left| \frac d { dt } \frac { \hat f (k) (t) } { \sqrt{ L(t) } } \right|^2
\]
by the Parseval identity.
Since
\[
	| \partial_t \bar f |^2
	=
	\| \partial_t \bar { \vecf } \|^2
	=
	\| \partial_t \vecf \|^2
	=
	\left\|
	\veckappa - \frac L { 2A } \vecnu \right\|^2
	=
	\left| \tilde \kappa - \frac L { 2A } I_{-1} \right|^2
\]
decays exponentially and uniformly in spatial variable as $ t \to \infty $,
we have
\[
	\sum_{ k \in \mathbb{Z} }
	\left| \frac d { dt } \frac { \hat f (k) (t) } { \sqrt{ L(t) } } \right|^2
	\leqq
	C e^{ - 2 \gamma t }
\]
for some $ C > 0 $ and $ \gamma > 0 $.
In particular,
\[
	\left| \frac d { dt } \frac { \hat f (1) (t) } { \sqrt { L(t) } } \right|^2
	\leqq
	C e^{ - 2 \gamma t }
	.
\]
Also, it is not difficult to see that
\[
	\left| \frac d { dt } \frac { \hat f (1) (t) } { \sqrt{ L(t) } } \right|^2
	=
	\left| \frac d { dt } r(t) \right|^2
	+
	4 \pi^2  r(t)^2
	\left| \frac d { dt } \left( \frac { \sigma (t) } { L(t) } \right) \right|^2
	.
\]
Since $ r(t) $ and $ L(t) $ converge exponentially to positive constants,
so does $ \sigma (t) $ to some $ \sigma_\infty \in \mathbb{R} / 2 \pi \mathbb{Z} $.
\item[{\rm (4)}]
We have
\begin{align*}
	\tilde { \vecf } ( \theta , t )
	= & \
	\vecc (t)
	+
	r(t) \left( \cos 2 \pi \theta , \sin 2 \pi \theta \right)
	+ \tilde { \vecrho } ( \theta , t )
	\\
	= & \
	\tilde { \vecf }_\infty ( \theta )
	+
	\vecc (t) - \vecc_\infty
	+
	( r(t) - r_\infty )
	\left( \cos 2 \pi \theta , \sin 2 \pi \theta \right)
	+ \tilde { \vecrho } ( \theta , t )
	,
\end{align*}
where
\[
	\tilde { \vecrho } ( \theta , t )
	=
	\vecrho ( L(t) \theta - \sigma (t) , t ) .
\]
Therefore the estimates for $ \vecc(t) $,
$ r(t) $,
and $ \vecrho ( \cdot , t ) $ yield
\[
	\| \tilde { \vecf } ( \cdot , t ) - \tilde { \vecf }_\infty \|_{ C^k ( \mathbb{R} / \mathbb{Z} ) }
	\leqq
	C_k e^{ - \tilde \gamma_k t }
	.
\]
\item[{\rm (5)}]
The above estimate implies that $ \mathrm{Im} \tilde { \vecf }( \cdot , t ) $ is the boundary of a bounded domain $ \Omega (t) $ when $ t $ is sufficiently large.
Since $ \tilde \kappa $ converges to $ 0 $ uniformly,
and since $ L $ goes to a positive constant $ L_\infty $ as $ t \to \infty $ uniformly in $s$,
\[
	\kappa = \frac { 2 \pi } L + \tilde \kappa
\]
is strictly positive for large $ t $.
Consequently $ \partial \Omega (t) $ is a strictly convex curve.
\item[{\rm (6)}]
Let $ D_r ( \vecc ) $ be the closed disk with center $ \vecc $ and the radius $ r $.
When $ t $ is sufficiently large,
\begin{align*}
	&
	d_H ( \overline { \Omega(t) } , D_{ r_\infty } ( \vecc_\infty ) )
	\\
	& \quad
	\leqq
	d_H ( \overline { \Omega(t) } , D_{ r (t) } ( \vecc (t) ) )
	+
	d_H ( D_{ r(t) } ( \vecc (t) ) , D_{ r(t) } ( \vecc_\infty ) )
	+
	d_H ( D_{ r(t) } ( \vecc_\infty ) , B_{ r_\infty } ( \vecc_\infty ) )
	\\
	& \quad
	\leqq
	C \left( 
	\| \tilde { \vecrho } ( \cdot , t ) \|_{ C^0 ( \mathbb{R} / \mathbb{Z} ) }
	+
	\| \vecc (t) - \vecc_\infty \|
	+
	| r(t) - r_\infty |
	\right)
	\\
	& \quad
	\leqq
	C e^{ - \gamma t }
\end{align*}
for some $ C > 0 $ and $ \gamma > 0 $.
\item[{\rm (7)}]
Clearly we have
\[
A ( \vecb - \vecc )
	=
	\iint_{ \Omega (t) } ( \vecx - \vecc ) \, d \vecx.
\]
We define 
\[
\vecb = ( b_1 , b_2 ), \quad \vecc = ( c_1 , c_2 ), \quad
 b = b_1 + i b_2 ,  \quad c = c_1 + i c_2.
\]
From the divergence theorem, we have
\begin{align*}
	A ( b - c )
	= & \
	\iint_{ \Omega (t) } \{ ( x_1 - c_1 ) + i ( x_2 - c_2 ) \} d \vecx
	\\
	= & \
	\frac 12 \iint_{ \Omega (t) }
	\mathrm{div} ( \left( ( x_1 - c_1 )^2 , i ( x_2 - c_2 )^2 \right) d \vecx
	\\
	= & \
	- \frac 12 \int_0^L
	\left( ( f_1 - c_1 )^2 , i ( f_2 - c_2 )^2 \right) \cdot \vecnu \, ds
	\\
	= & \
	- \frac 12 \int_0^L
	\left\{ ( f_1 - c_1 )^2 ( - \partial_s f_2 )
	+
	i ( f_2 - c_2 )^2 \partial_s f_1 \right\} ds.
\end{align*}
Since, for $j = 1 , \, 2 $, 
\[
	\int_0^L ( f_j - c_j )^2 \partial_s f_j \, ds = 0,
\]
we obtain 
\begin{align*}
	A ( b - c )
	= & \
	- \frac i 2 \int_0^L
	\left\{
	( f_1 - c_1 )^2 ( \partial_s f_1 + i \partial_s f_2 )
	+
	( f_2 - c_2 )^2 ( \partial_s f_1 + i \partial_s f_2 )
	\right\} ds
	\\
	= & \
	- \frac i 2 \int_0^L | f - c |^2 \partial_s f \, ds.
\end{align*}
It holds that 
\[
	| f - c |^2
	=
	\left| r e^{ \frac { 2 \pi i ( s + \sigma ) } L } + \rho \right|^2
	=
	r^2 + 2 r \Re \rho e^{ \frac { 2 \pi i ( s + \sigma ) } L } + | \rho |^2
\]
and 
$ r $ is independent of $ s $. Hence we show
\[
	A ( b - c )
	=
	- \frac i 2 \int_0^L
	\left\{
	2 r \Re \rho e^{ \frac { 2 \pi i ( s + \sigma ) } L } + | \rho |^2
	\right\}
	\partial_s f \, ds.
\]
Therefore we have
\[
        \| A ( \vecb - \vecc ) \|
        \leqq
	| A ( b - c ) |
	\leqq
	\frac 12 \int_0^L
	\left( 2r | \rho | + | \rho |^2 \right) ds
	\leqq
	C e^{ - \gamma t }.
\]
\end{itemize}
Thus we have shown each of the claims in the theorem.
\qed
\end{proof}
\subsection{The area-preserving curvature flow}
\par
It is well-known that if the initial curve is convex,
then any solution of the area-preserving curvature flow
\be
	\partial_t \vecf =  \tilde \kappa  \vecnu
	\label{area-preserving_flow}
\ee
converges to a round circle as $ t \to \infty $ as
proved by Gage \cite{G}.
In this subsection,
we give a proof of this fact without the convexity assumption {\it assuming} the global existence.
\par
If the initial curve is convex,
then the convexity remains for all $ t > 0 $ by the maximal principle.
In this case the exponential decay of $ I_{-1} (t) $ is easily derived from Gage's inequality.
First we show the decay without the convexity assumption.
\begin{thm}
Assume that $ \vecf $ is a global solution of {\rm \pref{area-preserving_flow}} such that the initial rotation number is $ 1 $.
Then $ \mathrm{Im} \vecf $ converges to a circle exponentially as $ t \to \infty $ in the sense that
\begin{gather}
	0 \leqq L(t)^2 - 4 \pi A (t)
	\leqq
	( L(0)^2 - 4 \pi A(0) ) \exp \left( - \frac { 16 \pi^2 } { L(0)^2 } t \right)
	,
	\\
	\left| L(t) - 2 \sqrt{ \pi A(0) } \right|
	\leqq
	\frac { L(0)^2 - 4 \pi A(0) }
	{ 4 \sqrt{ \pi A(0) } }
	\exp \left( - \frac { 16 \pi^2 } { L(0)^2 } t \right) .
\end{gather}
\label{Theorem7}
\end{thm}
\begin{proof}
Since
\[
	\frac { d L } { dt }
	=
	- \int_0^L \partial_t \vecf \cdot \veckappa ds
	=
	- \int_0^L \tilde \kappa^2 ds
	,
\]
we have
\[
	L \leqq L (0) .
\]
From this,
the area-preserving property,
and Theorem \ref{Theorem2},
we have
\begin{align*}
	\frac d { dt } ( L^2 - 4 \pi A )
	= & \
	2L \frac { dL } { dt }
	=
	- 2L \int_0^L \tilde \kappa^2 ds
	=
	- 2 I_0
	\\
	\leqq & \
	- \frac { 16 \pi^2 } { L^2 } ( L^2 - 4 \pi A )
	\leqq
	- \frac { 16 \pi^2 } { L(0)^2 } ( L^2 - 4 \pi A )
	.
\end{align*}
Therefore
\[
	L^2 - 4 \pi A
	\leqq
	( L(0)^2 - 4 \pi A(0) ) \exp \left( - \frac { 16 \pi^2 } { L(0)^2 } t \right)
	.
\]
Hence we have $ \displaystyle{ \lim_{ t \to \infty } L = 2 \sqrt{ \pi A(0) } } $,
and
\[
	\left| L - 2 \sqrt{ \pi A(0) } \right|
	\leqq
	\frac { L^2 - 4 \pi A } { L + 2 \sqrt { \pi A(0) } }
	\leqq
	\frac { L(0)^2 - 4 \pi A(0) }
	{ 4 \sqrt{ \pi A(0) } }
	\exp \left( - \frac { 16 \pi^2 } { L(0)^2 } t \right) .
\]
\qed
\end{proof}
\par
For this flow the limit value of $ L $ and the decay rate are given explicitly from the initial data.
\par
We can prove the following theorem similarly to the proofs of Theorems \ref{Theorem5}--\ref{Theorem6} using Theorem \ref{Theorem7} instead of Theorem \ref{Theorem4}.
\begin{thm}
The claims (1)--(7) in Theorem \ref{Theorem6} hold also for global solutions of the area-preserving flow.
\label{Theorem8}
\end{thm}
\par\noindent
{\bf Acknowledgment}.
The first author was partly supported by Grant-in-Aid for Scientific Research (C) 
(17K05310),
Japan Society for the Promotion Science.
The authors express their appreciation to Professor Shigetoshi Yazaki and Professor Tetsuya Ishiwata for sharing information of related articles,
and for discussions.
The authors also would like to express their gratitude to Professor Neal Bez for English language editing.

\end{document}